\documentclass[reqno, 12pt]{amsart}

\usepackage{amsthm,amssymb,amstext,amscd,amsfonts,amsbsy,amsxtra,latexsym, amsmath}
\usepackage{fullpage}
\usepackage[latin1]{inputenc}%

\newcommand{\field}[1]{\mathbb{#1}}

\newcommand{\Z}{\field{Z}}
\newcommand{\HHH}{\mathcal{H}}
\newcommand{\R}{\field{R}}
\newcommand{\C}{\field{C}}
\newcommand{\g}{\gamma}
\newcommand{\G}{\Gamma}
\newcommand{\HH}{\mathfrak{H}}

\def\({\left(}
\def\){\right)}
\newcommand{\nbd}{\nobreakdash-\hspace{0pt}}

\def\ca{{\mathfrak a}}
\def\cb{{\mathfrak b}}

\theoremstyle{plain}
\newtheorem{theorem}{Theorem}
\newtheorem*{theorem*}{Theorem}
\newtheorem{lemma}[theorem]{\textbf{Lemma}}
\newtheorem{proposition}[theorem]{\textbf{Proposition}}
\newtheorem{corollary}[theorem]{\textbf{Corollary}}
\newtheorem*{conjecture*}{Conjecture}

\theoremstyle{definition}
\newtheorem{definition}[theorem]{Definition}

\theoremstyle{remark}
\newtheorem{remark}{Remark}

\numberwithin{theorem}{section} \numberwithin{equation}{section}

\renewenvironment{proof}[1][Proof]{\begin{trivlist}
\item[\hskip \labelsep {\bfseries #1:}]}{\qed\end{trivlist}}

\begin{document}
\newcommand{\Q}{{\mathbb Q}}
\newcommand{\sgn}{\mathrm{sgn}}

\title{Mock period functions, sesquiharmonic Maass forms, and non-critical
values of $L$-functions}
\author{Kathrin Bringmann}
\address{Mathematical Institute, University of Cologne, Weyertal 86-90, 50931 Cologne, Germany}
\email{kbringma@math.uni-koeln.de}
\author{Nikolaos Diamantis}
\address{School of Mathematical Science, University of Nottingham,
Nottingham NG7 2RD, UK \\
Max-Planck-Institut f\"ur Mathematik,
Vivatsgasse 7,
53111 Bonn,
Germany}
\email{diamant@mpim-bonn.mpg.de}
\author{Martin Raum}
\address{Max-Planck-Institut f\"ur Mathematik,
Vivatsgasse 7,
53111 Bonn,
Germany}
\email{MRaum@mpim-bonn.mpg.de}
\thanks{The research of the first author was supported by the Alfried 
Krupp Prize for Young University Teachers of the Krupp Foundation. The
second author is supported by the Max-Planck Institut for Mathematics, 
Bonn.
The third author holds a scholarship from the Max-Planck Society of 
Germany.}

\date{}

\begin{abstract} We introduce a new technique of
completion for $1$\nobreakdash-\hspace{1pt}cohomology which parallels the
corresponding technique in the
theory of mock modular forms. This technique is applied in the
context of non-critical values of $L$-functions of $\mathrm{GL}(2)$ cusp forms.
We prove that a generating series of non\nobreakdash-\hspace{0em}critical
values can be
interpreted as a mock period function we define in analogy with period
polynomials. Further,
we prove that non-critical values can be encoded into a sesquiharmonic Maass
form. Finally, we formulate and
prove an Eichler-Shimura-type isomorphism for the space of mock period
functions.
\end{abstract}


\maketitle

\section{Introduction}
\noindent
In this work, we establish a connection between two seemingly
disparate topics and techniques: mock modular forms (holomorphic parts
of harmonic Maass forms) and
non-critical
values of $L$-functions of cusp forms. To describe this connection, we
first outline each of these topics and some of the corresponding questions
that arise.

A very fruitful technique that has recently emerged in the broader area
of
automorphic forms and its arithmetic applications is based on
``completing"
a holomorphic but not quite automorphic form into a harmonic Maass
form by addition of a suitable non-holomorphic function. This method
originates in its modern form in Zwegers' PhD thesis  \cite{Zw}.
  Zwegers completed all of Ramanujan's mock theta
functions introduced by Ramanujan in
his famous last letter to Hardy \cite{Ra}, including
\begin{equation*}
f(q):=1+\sum_{n=1}^{\infty}\frac{q^{n^2}}{(1+q)^2(1+q^2)^2\cdots (1+q^n)^2}.
\end{equation*}
To be more precise, Zwegers found a (purely) non-holomorphic
function
\begin{equation} \label{nonholint}
N_f(z):= \int_{- \overline{z}}^{i\infty} \frac{\Theta_f(w)}{\sqrt{z+w}}dw,
\end{equation}
where $\Theta_f$ is some explicit weight $\frac32$ cuspidal theta
function, so
that $$f(q)+N_f(z)$$
transforms like an automorphic form of weight ``dual"
to that of $f$, i.e., of weight $\frac12$ in our case (throughout  we write $q:=e^{2 \pi i z}$).
Such  completions proved to be useful in obtaining information for
the original function ($f$ in our context), including exact
formulas for Fourier coefficients, made use of, e.g., in the proof in \cite{BO1} of the
Andrews-Dragonette Conjecture \cite{An,Dr}.
On the other hand, one can also reverse
 the question and start with a
modular form,  define
an integral $N$ resembling the one in (\ref{nonholint}) and find a
holomorphic function $F$ such that
$N+F$ transforms like a modular form. Such
``lifts" were constructed for cusp forms of weight $\frac12$ in terms of
combinatorial series by the first author, Folsom, and Ono \cite{BFO} and
by the first
 author and Ono for general cusp forms \cite{BO4}. Recently, also lifts for
non-cusp forms were found \cite{DIT}.
Obstructions to modularity occuring from functions like $f$
may also be viewed in terms of critical values of $L$-functions
\cite{BGKO} in a way
we will describe later.

We next introduce the second topic, non-critical values of
$L$-functions. We will first outline the background concerning general
values of $L$-functions and critical values. Let
$f$ be an element of $S_k$, the space of cusp forms of weight $k \in
2\mathbb N$ for $\mathrm{SL}_2(\mathbb Z)$,
and
let $L_f(s)$ denote its $L$\nbd function. Special values of $L$-functions
have
been the focus of intense research in arithmetic algebraic geometry and
analytic number theory, because they provide deep insight to $f$ and
associated arithmetic
and geometric objects. Several of the outstanding conjectures in number
theory are related to special values of $L$-functions, e.g. the
ones posed by
Birch-Swinnerton-Dyer, Beilinson and Bloch-Kato
(see, for example, \cite{KZ}).  In particular, they are commonly interpreted as regulators in $K$-theory \cite{Sch88}.

Among the special values, more is known about the \textit{critical} values
which, for our purposes, are $L_f(1), L_f(2), \dots ,L_f(k-1)$
(see \cite{De, KZ} for an intrinsic characterization). For instance,
Manin's Periods Theorem \cite{M} implies that, when $f$ is an eigenform of
the Hecke operators, its critical values are algebraic linear combinations
of two constants depending only on $f$. This result was established by
incorporating a ``generating function" of the critical values into a
cohomology which has a rational  structure.
The generating
function is the \textit{period polynomial}
$$
r_f(X):=\int_0^{i\infty} f(w)(w-X)^{k-2} dw
\text{,}
$$
and each of its coefficients is an explicit multiple of a critical values
of $L_f(s)$ (see Lemma \ref{r_f} for the precise statement).

The period polynomial of $f$ satisfies the \textit{Eichler-Shimura
relations}:
$$
r_f|_{2-k}(1+S)=r_f\Big|_{2-k}\left(1+
U+U^2\right)=0 \qquad \text{
with
$S:=\left(\begin{smallmatrix} 0 & -1\\ 1 & 0\end{smallmatrix}\right)$,
$U:=\left(\begin{smallmatrix} 1 & -1\\ 1 & 0\end{smallmatrix}\right)$}
$$
in terms of the action $|_m$ on $G: \HH \to \C$ defined for each $m \in 2
\mathbb Z$ by
\begin{equation*}\label{|}
G|_m\g(X):=G(\g X) (cX+d)^{-m} \qquad \text{for $\g=\left (
\begin{smallmatrix} * & * \\
c & d \end{smallmatrix}\right ) \in \mathrm{SL}_2(\R)$.}
\end{equation*}
Because of the importance of these Eichler-Shimura relations, the
space $V_{k-2}$ of all polynomials of degree at most $k-2$ satisfying
them
has been
studied independently. It is called the \textit{space of period polynomials}
and is denoted by $W_{k-2}$.

Non-critical values are much less understood and there are
even some ``negative" results such as that of Koblitz \cite{Ko}, asserting
that, in a strong sense, there can not be a Period Theorem for
non-critical values. In any case, it is generally expected that the
algebraic
structure of such values is more complicated than that of critical values.
Nevertheless, in \cite{CD} it is shown that it is possible to
define ``generating series"  of non-critical values, which can further be
incorporated into a cohomology similar  to the Eichler
cohomology.  This fits into the philosophy of Manin's
\cite{M2} and Goldfeld's \cite{Go} cohomological interpretation of
values and derivatives of $L$-functions, respectively.
The generating series is a function $r_{f, 2}$ on the
Poincar\'e upper-half plane $\HH$ given by
$$
r_{f, 2}(z):=\int_0^{i\infty} \frac{F_f(w)}{(wz-1)^k} dw,
$$
where $F_f$ is the \textit{Eichler integral}
associated to $f$
$$
F_f(z):=\int_{z}^{i\infty} f(w)(w-z)^{k-2} dw.
$$
The function $r_{f, 2}$ is the direct counterpart of the period
polynomial $r_f$ associated to critical values.
The non-critical values are obtained  from $r_{f, 2}$  as ``Taylor
coefficients" of
$r_{f, 2}$ (see Lemma \ref{r_{f,2}}),  just as critical values are
retrieved as coefficients of the period polynomial $r_f$.  The
ambient space of functions consists of  harmonic functions rather
than
polynomials and the action is $|_k$ instead of $|_{2-k}.$

The \textit{first} link between the aforementioned two topics emerges as we
use techniques from the theory of mock modular forms to
intrinsically
interpret the constructions that were associated to
non-critical values in \cite{CD}. Those constructions
were in some respects ad hoc and
not as intrinsic as those relating to critical values.
 For example,  whereas the period polynomial is
expressed as a constant multiple of
$$
F_f|_{2-k}(S-1),
$$
the generating function $r_{f, 2}(z)$ has an analogous expression
only up
to an explicit ``correction term".  That problem would seem to be
insurmountable, because $r_{f, 2}(z)$ is not invariant under $S$.

However, in this paper we show that it is
exactly thanks to
 the ``correction term" that our generating
function $r_{f, 2}$ can be completed into a function which belongs
to a natural  analogue of  the space of period polynomials $W_{k-2}$.
We
show that an appropriate counterpart of
 $$
W_{k-2}:=\{P
\in V_{k-2}; P|_{2-k}(1+S)=P|_{2-k}\left(1+U+U^2\right)=0\}$$
\rm is
$$
W_{k, 2}:=
\left\{\mathcal{P}: \mathfrak H \to \C;\,
\xi_k(\mathcal{P}) \in V_{k-2};  \,
\mathcal{P}|_k(1+S)=\mathcal{P}|_k\left(1+U+U^2\right)=0\right\}.
$$
Here, $\xi_{k}$ is a key operator in the theory of mock modular
forms defined, for $y:=$Im$(z)$ by
\[
\xi_{k}:=2iy^{k}\overline{\frac{d}{d\overline{z}}}.
\]
Our first main result then is
\begin{theorem}\label{W_k,2intro}
Let $k \in 2 \mathbb N$ and
$f$  a weight $k$ cusp form.  Then the function
\begin{equation*}\label{rf2}
\widehat r_{f, 2}(z):=r_{f, 2}(z)-
 \int_{-\overline z}^{i \infty} \frac{r_f(w)}{(w+z)^k}dw
\end{equation*}
belongs to  the space  $W_{k, 2}$.
\end{theorem}
Theorem \ref{rf2} suggests the name \textit{mock period
function}
for $r_{f, 2}$ (see Definition \ref{definemock})

The completion  of $r_{f, 2}$  by a purely non-holomorphic
term does not
cause us to lose
information about non-critical values, because it only introduces
critical values (see Lemma \ref{tilde2rf}),
which from our viewpoint can be thought of as understood.

The \textit{second} link between the two main subjects of the
paper amounts to a technique that allows us to encode information about
the
mock period function of $f \in S_k$ into a certain ``higher order" version
of
harmonic Maass forms. This is the direct analogue of a recent
result
proved for
critical values by the first author, Guerzhoy, Kent, and Ono (Theorem 1.1
of \cite{BGKO}) and in a different guise earlier in \cite{Fay}:
\begin{theorem}\label{periodPoincareintro1} (\cite{Fay, BGKO})
For each $f \in S_k$, there is a harmonic
Maass form $M_f$
with holomorphic part $M_f^+$, such that
\begin{equation*}\label{moddef}
r_f(-z)=M_f^+|_{2-k}(1-S).
\end{equation*}
\end{theorem}
The authors further use similar techniques to establish a
structure theorem
for $W_{k-2}$ (Theorem 1.2 of \cite{BGKO}).

The first step of our approach towards establishing the counterpart of
Theorem \ref{periodPoincareintro1} for non-critical values is to identify
the objects taking the role played by harmonic Maass forms in \cite{BGKO}.
The class of these objects  is formed by
 \textit{sesquiharmonic  Maass forms} (see Definition \ref{sesquiharm}).
Sesquiharmonic Maass form  are  natural higher
order versions of harmonic Maass forms, the first example of which has
appeared in a
different  context \cite{DI, DIT}. (See also \cite{Br1, Br2, Br3} for an
earlier application
of the underlying method).
The main difference of sesquiharmonic to harmonic Maass forms is that the
latter are annihilated by
the weight $k$ Laplace-operator
\[
\Delta_k:=-y^2\left(\frac{\partial^2}{\partial
x^2}+\frac{\partial^2}{\partial
y^2}\right)+iky\left(\frac{\partial}{\partial x}+i\frac{\partial}{\partial
y}\right)
\text{,}
\]
whereas sesquiharmonic Maass forms are annihilated by
\[
\Delta_{k, 2}:=\Delta_{2-k}\circ
\xi_k=-\xi_k\circ\xi_{2-k}\circ\xi_k=\xi_k\circ\Delta_k.
\]
 In Section \ref{Sesquisection}, we will show that we can isolate a
``harmonic" piece from each sesquiharmonic Maass, paralleling the way
we can isolate a ``holomorphic" piece from each harmonic Maass form.
This construction allows us to formulate and prove the analogue of
Theorem \ref{periodPoincareintro1}:
\begin{theorem}\label{periodPoincareintro}
For each $f \in S_k$, there is a sesquiharmonic
Maass form $M_{f, 2}$
with harmonic part $M_{f, 2}^{+-}$, such  that
\[
\widehat{r}_{f, 2}(z)= M_{f, 2}^{+-}(z)\Big|_k(S-1).
\]
\end{theorem}
The above two techniques we just described  can be considered as a
new version of the ``completion" method, this time applied to the level of
$1$-cohomology.

The \textit{third} main result and technique of this paper is a
mock Eichler-Shimura isomorphism for $W_{k, 2}.$ The
classical Eichler-Shimura isomorphism ``parametrizes" $W_{k-2}$
in terms of cusp forms. It can be summarized as:
\begin{theorem}\label{ESisom1} (e.g., \cite{KoZ}) Every $P \in W_{k-2}$ can
be written as
$$
P(X)=r_f(X)+r_g(-X)+a|_{2-k}(S-1)
$$
for  unique  $f, g \in S_k$ and $a \in \C$.
\end{theorem}
In Section \ref{ESchar},
we show that $W_{k, 2}$ can be
``parametrised" by cusp forms  in a very similar fashion:
\begin{theorem}\label{ESisom2} Every $P \in W_{k, 2}$ can
be written as
$$P=\widehat r_{f, 2}+\widehat r^*_{g, 2}+a F|_k(S-1)
$$
for   unique  $f, g \in S_k$ and an $a \in \C$.
Here, $F$ is an element of an appropriate space of functions on $\HH$
and $\widehat r^*_{g, 2}$ is a period function associated $r_g(-X)$. (They
will be defined precisely in Section \ref{ESchar}).
\end{theorem}
The construction of $\widehat r^*_{g, 2}$ is of independent
interest and involves (regularized) integrals  (see Section \ref{ESchar}).
 Some of the
techniques are related to the
theory of periods of weakly holomorphic forms as studied by Fricke \cite{Fr}.

It is surprising that
pairs of cusp forms suffice for this Mock Eichler-Shimura
isomorphism just as they suffice for the classical
Eichler-Shimura isomorphism.  A priori,  the spaces $W_{k-2}$ and
$W_{k, 2}$ appear to be very different, especially since, as shown here,
they are associated with critical and non-critical values respectively,
which are expected to have completely different behaviour.

In the final section we interpret our two first main results
cohomologically (Theorem \ref{periodPoincare'}) in order to highlight the essential
similarity of the construction we associate here to non-critical
values with the corresponding setting for critical values. Since we
have an entirely analogous reformulation (see \eqref{ESBGKO}) of the
Eichler-Shimura theory and the
results of \cite{BGKO}, Theorem \ref{periodPoincare'} justifies the claim
that our constructions form the non-critical value counterpart of the
corresponding results in the case of \textit{critical} values of
$L$-functions.

A suggestive comparison of this cohomological interpretation with
Hida's evidence for a possible description of non-critical values in
terms of non-top degree cohomology (cf. \cite{Hi}) might also be made. We
intend to
return to possible explicit connections with Hida's construction in a
future work.

\vspace{0.5em}\noindent
\textit{Acknowledgments}: To be entered after the referee's report is received.

\section{Cusp forms and periods associated to their $L$-values}
\label{periods}
 Set $\G:= {\mathrm SL}_{2}(\Z)$.
Let $f(z)=\sum_{n=1}^{\infty} a(n) q^n$
($q=e^{2 \pi i z}$)
be a cusp form of weight $k$ for $\Gamma$.
Further let $L_f(s)$ be the entire function obtained by analytic continuation of
the series $L_f(s)=\sum_{n=1}^{\infty} a(n)/n^s$ originally defined in an appropriate
right half plane.

In the Eichler-Shimura-Manin
theory one associates to $f$ an Eichler integral $F_f: \HH \to \C$
and a period
polynomial  $r_f: \C \to \C$
as follows:
\begin{eqnarray*}
F_f(z)&:=&\int_{z}^{i\infty} f(w)(w-z)^{k-2} dw,\\
r_f(z)&:=&\int_0^{i\infty} f(w)(w-z)^{k-2}dw.
\end{eqnarray*}
These objects are connected to each other and intimately related to
critical values of $L_f(s)$ (see e.g.  \cite{KoZ}, Section 1.1):
$L_f(1), \dots, L_f(k-1)$.
\begin{lemma} \label{r_f} For every $f \in S_k,$  we have
 \begin{eqnarray*}
 F_f|_{2-k}(1-S)&=&r_f, \\
r_f(z)&=& -\frac{(k-2)!}{(2 \pi i)^{k-1}}
\sum_{n=0}^{k-2}
\frac{L_f(n+1)}{(k-2-n)!}(2 \pi i z)^{k-2-n}.
\end{eqnarray*}
\end{lemma}
We shall consider the analogues of $F_f$ and $r_f$ yielding non-critical
values of $L_f(s)$. Set
\begin{eqnarray*}
F_{f, 2}(z)&:=&\int_{-\overline{z}}^{i\infty} \frac{F_f(w)}{(w+z)^k} dw,\\
r_{f, 2}(z)&:=&\left.\left( \int_0^{i\infty} \frac{F_f(w)}{(w+z)^k} dw \right) \right
|_k
S=\int_0^{i\infty} \frac{F_f(w)}{(wz-1)^k} dw.
\end{eqnarray*}
The function $r_{f, 2}$ is not a polynomial, but the next lemma, proved in
\cite{CD}, shows that we
can still retrieve values of $L$-functions of $f$ as its ``Taylor
coefficients at $0$". It also explains the reason for letting $S$ act on
the integral in the definition of $r_{f, 2}$ in an apparent disanalogy to
$r_f$:
\begin{lemma} \label{r_{f,2}} For every $f \in S_k$ and
$m \in \mathbb{N}$, we have
\begin{equation*} \lim_{z \to 0^+} \frac{d^m }{dz^m}\left( r_{f, 2}(z) \right)
=i^{k+m}\frac{(m+k-1)! m!}{(k-1) (2 \pi)^{m+k}}
L_f(k+m).
\end{equation*}
\end{lemma}
In \cite{CD}, it is also proved that $F_{f, 2}$ and $r_{f, 2}$ are linked
in
a
way that parallels the link between $F_f$ and $r_f$. For our purposes, we
will need a reformulation of that result:
\begin{proposition} \label{SuperM} For every $f \in S_k,$ we have
\begin{equation} \label{SuperMformula}
\left. F_{f, 2}\right|_k(S-1)=
r_{f, 2}-\widetilde r_{f, 2}
\end{equation}
with
$$
\widetilde r_{f, 2}(z):=
\int_{-\overline z}^{i\infty} \frac{r_f(w)}{(w+z)^k}dw.$$
\end{proposition}
\begin{proof}
>From the proof of Theorem  3  of \cite{CD}, it follows  that
$$
\left.F_{f, 2}(z)\right|_k(S-1)= r_{f, 2}(z) +
\left.\left ( \int_{-\overline z}^{0}
\frac{r_f(w)}{(w+z)^k}dw \right ) \right|_k S.
$$
\rm
The last term may now easily be  simplified  using that
$r_f \in W_{k-2}$.
\end{proof}
The correction term
$\widetilde r_{f, 2}$ may be explicitly expressed in terms of {\emph critical} values, and it does not affect the analogy with the relation between
$F_f$ and $r_f$.
\begin{lemma} \label{tilde2rf}
For all $f \in S_k$,
\begin{equation*} \label{tilde2rfform}
\widetilde r_{f, 2}(z)=
-(k-2)! \sum_{n=0}^{k-2}\sum_{\ell=0}^{k-2-n}
\frac{L_f(n+1)}{\ell! (k-2-n-\ell)!(1+n+\ell)}(-4 \pi i z)^{\ell} (-4 \pi
y)^{-1-n-\ell}.
\end{equation*}
\end{lemma}
\begin{remark} We note that all of the exponents of $y$ are
negative, thus $\widetilde r_{f, 2}$ is a purely non-holomorphic
function.
\end{remark}
\begin{proof}
>From Lemma \ref{r_f},
\begin{align*}
\int_{-\overline z}^{i\infty} \frac{r_f(w)}{(w+z)^k}dw
&=
 (k-2)!
\sum_{n=0}^{k-2} i^{-n+1}
\frac{L_f(n+1)}{(2 \pi)^{n+1}(k-2-n)!}   \int_{-\overline z}^{i\infty} \frac{w^{k-2-n}}{(w+z)^k}dw.
 \end{align*}
Making the change of variable $w \to w-z$ and then using the Binomial
Theorem, we obtain that
 the  integral equals
 \begin{equation*}
   \sum_{\ell=0}^{k-2-n} \binom{k-2-n}{\ell}(-z)^\ell
\frac{(2iy)^{-1-n-\ell}}{1+n+\ell}.
\end{equation*}
This implies the result.
\end{proof}
Because of Lemma \ref{tilde2rf}, it is natural to complete $r_{f,2}$ by
substracting this ``lower-order" non-holomorphic function to
obtain
$$
\widehat r_{f, 2}:=r_{f, 2}-\widetilde r_{f, 2}.
$$ Lemma
\ref{r_{f,2}} and Proposition
\ref{SuperM} suggest, by comparison with Lemma \ref{r_f}, that
$\widehat
r_{f, 2}$ can be viewed as an analogue of the period polynomial
associated to non-critical values.
In the next section, we will show that this interpretation can be
formalized  in a way that justifies the name  \textit{mock
period function} for $r_{f, 2}$.

\section{Mock period functions}
One of the reasons that the theory of periods has been so successful in
proving important results about the values of $L$-functions is that they
satisfy relations that allow us to view them as elements of a space with a
rational structure. This space is, in effect, the first cohomology
group of Eichler cohomology.  However, to make the relation with
$L$-functions more immediate we will use the more
concrete formulation and notation of
\cite{KoZ}.
In the last section, we will give a cohomological interpretation of our
results.

For  $n \in \mathbb N$, let $V_n$ denote the space of polynomials of degree at
most $n$ acted upon by $|_{-n}$, and set
$$
W_n:=\left\{P \in V_n; P|_{-n}(1+S)=P|_{-n}\left(1+U+U^2\right)=0
\right\}.
$$
The period polynomial $r_f$ associated to  $f \in S_k$ belongs to
$W_{k-2}$ (cf. \cite{KoZ}).
According to the well-known Eichler-Shimura Isomorphism (cf. \cite{KoZ}
and the references therein), the polynomials
characterize the entire space.
\begin{theorem}\label{ES} (Eichler-Shimura Isomorphism) Let $k$ be an
even positive integer. Then for each $P \in W_{k-2}$ there exists a
unique
pair $(f, g) \in S_k \times S_k$ and  $c \in \C$ such that
$$
P(z)=r_f(z)+r_g(-z)+c\left(z^{k-2}-1\right).
$$
\end{theorem}
\begin{remark} Usually, the second term is written as
$\overline {r_g(\bar z)}$, that is
the polynomial obtained by replacing each coefficient of  the
polynomial  $r_g$ with its
conjugate. However, this may be rewritten as
\begin{equation} \label{conj}
\overline {r_g(\overline z)} =
\int_0^{i\infty}\overline{g(w)}(\overline{w} - z)^{k-2} d \overline{w}
=-
\int_0^{i\infty}\overline{g(-\overline{w})}(-w - z)^{k-2}
dw=-r_{g^c}(-z).
\end{equation}
Recall that  $g^c(z):=\overline{g(-\overline{z})} \in S_k$.
\end{remark}

We will show that there is a space similar to $W_{k-2}$
within which the
completed period-like functions $\widehat r_{f, 2}$
live.
 We first
recall the operator $\xi_k:=2iy^k\frac{\overline{d}}{d\overline{z}}$
 ($y:=$Im$(z)$).
This map satisfies $\xi_k(f|_k \g)=(\xi_k f)|_{2-k} \g$ for all
$\g \in \G$, and thus maps weight $k$ automorphic objects to
weight $2-k$ automorphic objects.
We then set
$$
W_{k, 2}:=
\left\{\mathcal{P}: \mathfrak H \to \C; \xi_k(\mathcal{P}) \in V_{k-2};
\mathcal{P}|_k\left(1+S\right)=\mathcal{P}|_k\left(1+U+U^2\right)=0 \right\}.
$$
This space consists not of polynomials but of functions which
become polynomials only after application of the
$\xi_k$-operator.

The next theorem explains in what sense $r_{k, 2}$ can be
considered a mock period function.
\begin{theorem} \label{W_k,2} Let $k \in 2 \mathbb N$ and $f
\in S_k$.
Then the function $\widehat r_{f, 2}$ is an element of
$W_{k, 2}$.
\end{theorem}
\begin{proof} The first condition follows from the identity
\begin{equation}\label{*}
\xi_k\Big(\widehat{r}_{f, 2}(z)\Big)=-2i y^k\overline{\frac{d}{d\overline{z}}\int_{-\overline{z}}^{i \infty } \frac{r_f(w)}{(w+z)^k}\
dw}
=(2i)^{1-k}r_{f^c}(z)\in V_{k-2},
\end{equation}
where for the last equality we used  (\ref{conj}).
The relation
$$
\left.\widehat{r}_{f, 2} \right|_k(1+S)=0
$$
follows directly from the identity in Proposition \ref{SuperM}.

To deduce the relation for $U$ we first note that
$F_{f, 2}|_k T  =F_{f, 2}$, which follows directly from
 $f(w+1)=f(w)$.
 Thus
$$
\left.
F_{f, 2}\right|_k(1-S)=\left.F_{f, 2}\right|_k(1-TS)=\left.F_{f, 2}\right|_k(1-U)
$$
and the claim follows from $U^3=1$.
\end{proof}
\begin{remark}\label{harmo}
It is immediate that, if $\xi_k(\mathcal{P}) \in
V_{k-2}$, then $\Delta_k(\mathcal{P})=-\xi_{2-k} \circ \xi_k(\mathcal{P})=0$, and thus
Theorem \ref{W_k,2} implies that $\widehat{r}_{f, 2}$ is harmonic.
\end{remark}

This theorem suggests the name \textit{mock period function} for
$r_{f, 2}$ as well as the more general
\noindent
\begin{definition} \label{definemock}
A holomorphic function $p_2: \HH\to\C$ is called a
\textit{mock
period function} if there exists a
 $\widetilde{p}_2 \in \oplus_{j=1}^{k-1}y^{-j}V_{k-2}$
 such
that
\[ p_2+\widetilde{p}_2\in W_{k, 2}.\]
\end{definition}

The Eichler-Shimura relations for $\widehat r_{f, 2}$ proved in Theorem
\ref{W_k,2}
are reflected in \textit{mock} Eichler-Shimura relations for
$r_{f, 2}$.

\begin{theorem}\label{mockperiod}
We have
\begin{align*}
r_{f, 2}(z)\Big|_k (1+S) &=\int_0^{i\infty}\frac{r_f(w)}{(w+z)^k}\ dw,\\
r_{f, 2}(z)\Big|_k \Big(1+U+U^2\Big)&=\int_{-1}^{i\infty}\frac{r_f(w)}{(w+z)^k}\
dw  +  \int_{-1}^{0}\frac{r_f|_{2-k}\widetilde{U}(w)}{(w+z)^k}\ dw
\end{align*}
with
$\widetilde{U}:=\left(\begin{smallmatrix}
-1 & -1 \\ 1 & 0\end{smallmatrix}\right)= SU^2S^{-1}$.
\end{theorem}
\begin{proof}
By \eqref{SuperMformula} and Theorem \ref{W_k,2} it
suffices to consider the action of $1+S$ and $1+U+U^2$
on $\widetilde{r}_{f, 2}$ only. Further, since $r_f\in W_{k-2}$,
we have
\begin{equation} \label{periodrel}
r_f\Big|_{2-k}(1+S)=r_f\Big|_{2-k}\Big(1+U+U^2\Big)=0.
\end{equation}
For the first identity we have by (\ref{periodrel})
\begin{align*}
\widetilde{r}_{f, 2}(z)\Big|_k S
&=z^{-k}\int_{\frac{1}{\overline{z}}}^{i\infty}\frac{r_f(w)}{\left(w-\frac{1}{z}\right)^k}\ dw
 \\
&=\left(\int_{-\overline{z}}^{i\infty}-\int_0^{i\infty}\right)\frac{r_f|_{2-k}S(w)}{(w+z)^k}\ dw
=-\widetilde{r}_{f, 2}(z)+\int_0^{i\infty}\frac{r_f(w)}{(w+z)^k}\ dw.
\end{align*}
To prove the second identity, we observe that (\ref{periodrel}) implies that
\begin{equation} \label{periodrel2}
r_f\Big|_{2-k}\Big(1+\widetilde{U}+\widetilde{U}^2
\Big)= 0.
\end{equation}
The change of variables $w \to \widetilde{U}w$
gives
\[
\widetilde{r}_{f, 2}(z)\Big|_k U=
\int_{-\overline{z}}^0\frac{r_f|_{2-k}\widetilde{U}(w)}{(z+w)^k}\
dw.
\]
Likewise, the change of variables $w \to \widetilde{U}^2 w$
yields
\[
\widetilde{r}_{f, 2}(z)\Big|_k U^2=
\int_{-\overline{z}}^{-1}\frac{r_f|_{2-k}\widetilde{U}^2(w)}{(w+z)^k}\ dw.
\]
Thus
\begin{multline*}
\widetilde{r}_{f, 2}(z)\Big|_k\Big(1+U+U^2\Big)
 =\int_{-\overline{z}}^{i\infty}\frac{r_f|_{2-k}
\left(1+\widetilde{U}+\widetilde{U}^2\right)(w)}{(w+z)^k}\ dw \\
-\int_0^{i\infty}\frac{r_f|_{2-k} \widetilde{U}(w)}{(z+w)^k}\ dw
-\int_{-1}^{i\infty}\frac{r_f|_{2-k} \widetilde{U}^2(w)}{(w+z)^k} dw.
\end{multline*}
Applying (\ref{periodrel2}) we obtain
the claim.
\end{proof}

\section{Sesquiharmonic Maass forms}\label{Sesquisection}\label{section3}
In this section, we introduce
new automorphic objects related
to non-critical values of $L$-functions.
\begin{definition}\label{sesquiharm}
A real-analytic function $\mathcal{F}: \HH\to\C$ is called a \textit{sesquiharmonic
Maass form of weight $k$} if the following conditions are satisfied:
\begin{enumerate}
\item[i)] We have for all $\gamma\in \G$ that
$\mathcal{F}|_k \gamma=\mathcal{F}$.
\item[ii)] We have that
$\Delta_{k, 2}\left(\mathcal{F}\right)=0$.
\item[iii)] The function
$\mathcal{F}$ has at most linear exponential growth at infinity.
\end{enumerate}
\end{definition}

\noindent
We denote the space of such functions by $H_{k, 2}$. The subspace
of
harmonic weak Maass forms, i.e., these sesquiharmonic forms $\mathcal{F}$ that satisfy
\[
\Delta_k(\mathcal{F})=-\xi_{2-k}\circ \xi_k(\mathcal{F})=0
\]
is denoted by $H_k$. Our definition in particular implies that
\[
\xi_k\left(H_{k, 2}\right)\subset H_{2-k}.
\]
 The holomorphic differential $D := \frac{1}{2 \pi i} \frac{d}{d z}$ plays a role originating in Bol's identity.  It is well-known that (see \cite{BF})
$$
\xi_{2-k}\left(H_{2-k}\right) \subset M_k^!, \qquad
D^{k-1}\left(H_{2-k}\right) \subset M_k^!.
$$
Here, $M_k^!$ denotes the space of weakly holomorphic modular
form, i.e., those meromorphic  modular forms whose poles may only lie at
the cusps. This suggests the following distinguished subspaces.
\begin{definition} \label{H^+} For $k \in 2 \mathbb N$, set
\begin{enumerate}
\item[i)] $H_{2-k}^+:=\{f \in H_{2-k}; D^{k-1}(f) \in S_k\}$ and
$H_{2-k}^-:=\{f \in H_{2-k}; \xi_{2-k}(f) \in S_k\}$,
\item[ii)] $H_{k, 2}^+:=\{f \in H_{k, 2}; \xi_{k}(f) \in H^+_{2-k}\}$.
\end{enumerate}
\end{definition}
Employing the theory of Poincar\'e   series, we will prove
that the restriction of $\xi_k$ on  $H_{k, 2}^{+}$ surjects onto
$H_{2-k}^+.$
In general, for functions $\varphi$ that are translation invariant,
we define the following Poincar\'e series
\begin{equation} \label{generic}
\mathcal{P}_k(\varphi; z):=\sum_{\gamma\in\Gamma_\infty\setminus
\G }\varphi\Big|_k\gamma(z)
\end{equation}
whenever this series converges absolutely.
 Here, $\G_{\infty}$ is the
set of translations in $\G$. For $k >2$, the classical
Poincar\'e series, spanning $S_k$ for $m>0$, are
in this notation
\[
P_k(m; z):=\mathcal{P}_k\left(q^m; z\right).
\]
 For all $m \in \mathbb Z \setminus\{0\}$, the \textit{Maass Poincar\'e
series} are defined by \cite{He}
\[
\field{P}_{k}(m, s; z):=\mathcal{P}_k\left(\varphi_{m, s}; z\right)
\]
with
\[
\varphi_{m, s}(z):=\mathcal{M}_s^k(4\pi my)e(mx),
\]
Here, $e(x):=e^{2 \pi i x}$ and
\[
\mathcal{M}_s^k(u):=|u|^{-\frac{k}{2}}M_{\sgn(u)\frac{k}{2},
s-\frac12}\big(|u|\big),
\]
where $M_{\nu,\mu}$ is the usual $M$-Whittaker function
 with the  integral representation
\begin{equation}\label{Mwhit}
M_{\mu,
\nu}(y)=y^{\nu+\frac12}e^{\frac{y}{2}}\frac{\Gamma(1+2\nu)}{\Gamma\left(\nu+\mu+\frac12\right)
\Gamma\left(\nu-\mu+\frac12\right)}\int_0^1t^{\nu+\mu-\frac12}
(1-t)^{\nu-\mu-\frac12}e^{-yt}\ dt
\end{equation}
 for $\text{Re}\left(\nu\pm
\mu+\frac12\right)>0$.
Using that  as $y \to 0$
\begin{eqnarray} \label{Mbound}
\mathcal{M}_s^k(y)= O \left(y^{\text{Re}(s)-\frac{k}{2}} \right),
\end{eqnarray}
 we see that
the series $\field{P}_{k}(m, s; z)$ converges absolutely for
Re$(s)>1$  and  satisfies
\begin{equation}\label{Laplace}
\Delta_k\left(\field{P}_{k}(m, s; z) \right)
= \left(s(1-s)+\frac14\left(k^2-2k \right)\right)\field{P}_{k}(m, s; z).
\end{equation}
In particular, the Poincar\'e series is annihilated for $s= \frac{k}{2}$ or
$s=1-\frac{k}{2}$ (depending on the range of absolute
convergence).
Moreover, for $m>0$ and $k\geq 2$,
we have
\begin{equation}\label{xipoincare}
D^{k-1}\left(\field{P}_{2-k}\left(m, \frac{k}{2} ; z\right)\right)=
-(k-1)!m^{k-1} P_k(m; z)
\end{equation}
(see, e.g. \cite{BKR}) and
\begin{equation}\label{dipoincare}
\xi_{2-k} \left(\field{P}_{2-k}\left(-m, \frac{k}{2} ; z\right)\right)=
(k-1)(4 \pi m)^{k-1} P_k(m; z)
\end{equation}
(see, e.g.  Theorem 1.1 (2) of \cite{BO4}).
This implies
$$\field{P}_{2-k}\left(m, \frac{k}{2}; z\right) \in H_{2-k}^+, \qquad
\field{P}_{2-k}\left(-m, \frac{k}{2}; z\right) \in H_{2-k}^-.
$$
In fact, the Poincar\'e series span the  respective spaces
$H_{2-k}^+$ and
$H_{2-k}^-$.
For the space $H_k^-$ this follows from Remark 3.10 of \cite{BF}. \rm
For the space $H_k^+$ one
may argue analogously
by using the flipping operator \cite{BKR}, which gives a bijection between
the two spaces.

For $k>0$, we then set
\[
\field{P}_{k, 2}(m; z):= \mathcal{P}_k\left(\psi_{m}; z\right)
\]
with \[
\psi_{m}(z):=\frac{d}{ds}\left[\mathcal{M}_s^k(4\pi my)\right]_{s=\frac{k}{2}} e(mx).
\]
Differentiation in $s$ only introduces logarithms
and thus,  using (\ref{Mbound}), we can easily see
that, for Re$(s)>1$ and for every $\epsilon>0$, the derivative is
$O(y^{\text{Re}(s)-\epsilon-k/2})$, and thus, as $y \to 0$, we find
$\psi_m(z)=O(y^{-\epsilon}).$
Thus for all nonzero integers $m$, and $k>0$,
$\field{P}_{k, 2}(m; z)$ is absolutely convergent.

One could further explicitly compute the Fourier expansion of
$\field{P}_{k, 2}$ but for the purposes of
 this paper,  this is not required.
\begin{theorem} \label{PoinTheorem}
For $m \in \mathbb N$, the function $\field{P}_{k, 2}(-m; z)$ is
an element of
$H_{k, 2}^{+}$ and satisfies:
\begin{eqnarray} \label{imagexi}
\xi_k\left(\field{P}_{k, 2}(-m; z)\right)
&=&(4\pi m)^{1-k}\field{P}_{2-k}\left(m, \frac{k}{2}; z\right),   \\\label{imageD}
D^{k-1}\circ \xi_k\left(\field{P}_{k, 2}(-m;
z)\right)&=&-(k-1)! (4 \pi )^{k-1} P_k(m; z).
\end{eqnarray}
In particular, the map
\[
\xi_k \text{: } H_{k, 2}^{+}\to H_{2-k}^{+}
\]
is surjective.
\end{theorem}
\begin{proof}
Due to the absolute convergence of the series, the
transformation law is satisfied by construction.

To verify the (at most) linear exponential growth
at infinity of $\field{P}_{k, 2}(m; z)$ we recall that $M_{\mu, \nu}$ has
at most linear exponential growth as $y \to \infty$ (cf. \cite{NIST},
(13.14.20)). We further note that this also holds for its derivative in
$s$ and thus $\psi_m(z)$ too, because differentiation in $s$ only
introduces logarithms. Therefore, since
Im$(\g y) \to 0$ as $y \to \infty$ whenever $\g \ne 1$, we have
\[
\field{P}_{k, 2}(m; z) \ll |\psi_m(z)|+y^{-\frac{k}{2}}\sum_{\g \in
\G_{\infty}
\backslash \G - \{1\}} \text{Im}(\g z)^{-\epsilon+\frac{k}{2}}.
\]
This together with the well-known polynomial growth of Eisenstein
series at the cusps implies the
claim.

To prove \eqref{imagexi} and \eqref{imageD}, and thus the annihilation
under
$\Delta_{k, 2}$, we first note that $\xi_k$ commutes with the group
action of $\G$ and therefore we only have to compute
\begin{align} \label{xiaction}
\nonumber&\quad
  \xi_k\left(\frac{d}{ds}\left[\mathcal{M}_s^k(-4\pi
my)e(-mx)\right]_{s=\frac{k}{2}}\right)
\\&=
  y^k(4\pi
m)\overline{q}^{-m}\frac{d}{ds}\left[\frac{d}{dy}\left[\mathcal{M}_{s+\frac{k}{2}}^k(-y)e^{-\frac{y}{2}}\right]_{y=4\pi
my}\right]_{s=0}.
\end{align}
 Notice that we do not need to conjugate the internal function because
upon
differentiation at $s=0$ we obtain a real function. \
The integral representation \eqref{Mwhit} implies for $y>0$
\[
\mathcal{M}_{s+\frac{k}{2}}^k(-y)
e^{-\frac{y}{2}}=\frac{y^s\Gamma(2s+k)}{\Gamma(s)\Gamma(s+k)}\int_0^1t^{s-1}(1-t)^{s+k-1}e^{-yt}\
dt
\]
which, in turn, gives that
\begin{align*}
&\quad
\frac{d}{dy}\left(\mathcal{M}_{s+\frac{k}{2}}^k(-y)
e^{-\frac{y}{2}}\right)
\\[4pt]&=
\frac{s}{y}\cdot
y^{-\frac{k}{2}}M_{-\frac{k}{2},
s+\frac{k}{2}-\frac12}(y)e^{-\frac{y}{2}}
-\frac{y^s\Gamma(2s+k)}{\Gamma(s)\Gamma(s+k)}
\int_0^1t^s(1-t)^{s+k-1}e^{-yt}\ dt
\\[4pt]&=s
y^{-\frac{k}{2}-1}M_{-\frac{k}{2},
s+\frac{k}{2}-\frac12}(y)e^{-\frac{y}{2}}
-\frac{s}{2s+k}  y^{-\frac{k}{2}-\frac12}
M_{\frac12-\frac{k}{2}, s+\frac{k}{2}}(y) e^{-\frac{y}{2}}.
\end{align*}
Differentiating with respect to $s$ and setting $s=0$ gives (\cite{Sl}, (2.5.2))
\[
y^{-\frac{k}{2}-1}e^{-\frac{y}{2}}\frac1k\left(kM_{-\frac{k}{2},
\frac{k}{2}-\frac12}(y)-\sqrt{y}M_{\frac12-\frac{k}{2},
\frac{k}{2}}(y)\right)
=y^{-\frac{k}{2}-1}e^{-\frac{y}{2}}M_{1-\frac{k}{2},
\frac{k}{2}-\frac12}(y)
=e^{-\frac{y}{2}}y^{-k}\mathcal{M}_{\frac{k}{2}}^{2-k}(y).
\]
Thus
\begin{equation*}
\xi_k\left(\frac{d}{ds}\left[\mathcal{M}_s^k(-4\pi
my)e(-mx)\right]_{s=\frac{k}{2}}\right)
 =(4\pi m)^{1-k}\mathcal{M}_{\frac{k}{2}}^{2-k}(4\pi my)e(mx),
\end{equation*}
which implies \eqref{imagexi}.
>From \eqref{imagexi} we may also deduce that $\Delta_{k, 2}\Big(
\field{P}_{k, 2}(m; z)\Big)=0.$
Equality \eqref{xipoincare} implies (\ref{imageD}).
Since, as mentioned above the functions $\field{P}_{2-k}(m, k/2; z)$
span $H^{+}_{2-k}$, \eqref{imagexi} implies the last assertion.
\end{proof}

Since we have a basis of $S_k$ consisting of Poincar\'e series,
Theorem \ref{PoinTheorem} implies
\begin{corollary}\label{surjective}
For $f\in S_k$ there exists $\mathcal{M}_{f,2}\in H_{k, 2}^{+}$ such
that
\[
D^{k-1}\circ \xi_k\left(\mathcal{M}_{f,2}\right)=f.
\]
\end{corollary}

To state and prove our second main theorem we analyze the Fourier
expansion of $\mathcal{F}$
in $H_{k, 2}^{+}$.
Since
$F:=\xi_k\left(\mathcal{F}\right)\in H_{2-k}^{+}$, it has a Fourier
expansion of
the form
\begin{equation*}\label{f.e.}
F(z)=
\sum_{n\geq 0}\widetilde{a}(n)q^n+\sum_{\substack{n \gg -\infty\\ n
\not=0}}\widetilde{b}(n)\Gamma(k-1, 4\pi
ny)q^{-n}
\end{equation*}
for some $\widetilde{a}(n), \widetilde{b}(n)\in \C$ and $\Gamma(s,y)$ the incomplete gamma function (see, for instance, \cite{BF}).
 The first summand is called the
\textit{holomorphic part} and  the
second the \textit{non-holomorphic part} of $F$, and we denote
them by $F^{+}$ and   $F^{-}$, respectively.
 A direct calculation   implies that for some
$a(n), b(n), c(n), d(0) \in \C$
\begin{equation}\label{superexp}
\mathcal{F}(z)=\sum_{n\gg -\infty}a(n)q^n+\sum_{n>0}
b(n)\Gamma(1-k, 4\pi ny) q^{-n}
+\sum_{\substack{ n \gg - \infty \\ n \not=0 }}
c(n)\mathbf{\Gamma}_{k-1}(4\pi ny)q^n
+ d(0)y^{1-k},
\end{equation}
where for $y>0$, we define
\begin{equation*}\label{Maasssplit}
\mathbf{\Gamma}_s(y):=\int_y^\infty\Gamma(s, t) t^{-s}e^t \frac{dt}{t}.
\end{equation*}
Similarly for $y<0$, we integrate from $-\infty$ instead of $\infty$.
 We call the first summand of the right hand side of \eqref{superexp} the
\textit{holomorphic part}, the
second the  \textit{harmonic part}, and the third the
\textit{non-harmonic part} of $\mathcal{F}$ and we  denote
them by $\mathcal{F}^{++}$, $\mathcal{F}^{+-}$, and $\mathcal{F}^{--}$ respectively.
We note that for $\mathcal{F}^{++} \not= 0, \mathcal{F}^{+-} \not= 0,$ and
$\mathcal{F}^{--} \not= 0$, we have
\begin{equation} \label{vanishxi}
\xi_{k}\left(\mathcal{F}^{++}\right)=0, \quad \xi_{k}\left(\mathcal{F}^{+-}\right)\not=0 \quad \xi_{k}\left(\mathcal{F}^{--}\right)\not=0,
\quad
  \xi_k\left(y^{1-k}\right) \ne 0
\text{,}
\end{equation}
\begin{equation}\label{vanishxixi}
\xi_{2-k}\circ \xi_{k}\left(\mathcal{F}^{+-}\right)=0, \quad \xi_{2-k}\circ \xi_{k}\left(\mathcal{F}^{--}\right)\not=0,
\quad
  \xi_{2 - k} \circ \xi_k \left(y^{1-k}\right) = 0
\text{,}
\end{equation}
\begin{equation}\label{vanishDxi}
D^{k-1}\circ \xi_{k}\left(\mathcal{F}^{+-}\right)\not=0, \quad
D^{k-1}\circ \xi_{k}\left(\mathcal{F}^{--}\right)=0, \quad
D^{k-1}\circ \xi_{k}\left(y^{1-k}\right)=0
.
\end{equation}

With this terminology and notation we have
\begin{theorem}\label{periodPoincare}
For $f \in S_k$, there is a $\mathcal{M}_{f, 2} \in
H_{k, 2}^{+}$ such that
$D^{k-1}\circ \xi_k\left(\mathcal{M}_{f, 2}\right)=
-
\frac{(k-2)!}{(4\pi)^{k-1}} f^c$ and
\[
\widehat{r}_{f, 2}(z)=
\mathcal{M}_{f, 2}^{+-}(z)\Big|_k(S-1).
\]
\end{theorem}
\begin{proof}
By equation \eqref{SuperMformula},
\[
\widehat{r}_{f, 2}= F_{f, 2}\Big|_k (S-1).
\]
By Corollary \ref{surjective}, there is a $\mathcal{M}_{f, 2}\in H_{k,
2}^{+}$ such that
\begin{equation} \label{Dx}
D^{k-1}\circ \xi_k\left( \mathcal{M}_{f, 2}\right)=
-
\frac{(k-2)!}{(4\pi)^{k-1}}
f^c.
\end{equation}
We claim that
\[
F_{f, 2}= \mathcal{M}_{f, 2}^{+, -}.
\]
A direct computation inserting the Fourier expansion of $f$ gives that
$F_{f, 2}(z)$ has a Fourier expansion of the
form
\[
\sum_n b(n)\Gamma(1-k, 4\pi ny) q^{-n}.
\]
Next
\begin{eqnarray*}
\xi_k\left(F_{f, 2}(z)\right)
=(2i)^{1-k}F^c_f(z)
&=&(2i)^{1-k}\int_{-\overline{z}}^{i\infty}\overline{f(w)}(z+\overline{w})^{k-2}\
d\overline{w} \\
&=&-(2i)^{1-k}\int_z^{i\infty}f^c(w)(z-w)^{k-2}
dw.
\end{eqnarray*}
 This implies that
\[
D^{k-1}\circ \xi_k\Big(F_{f, 2}\Big)=
-
\frac{(k-2)!}{(4\pi)^{k-1}}f^c.
\]
Thus by (\ref{Dx}),
\begin{equation*}\label{anni}
D^{k-1}\circ
\xi_k\Big(F_{f, 2}-\mathcal{M}_{f, 2}\Big)=0.
\end{equation*}
By (\ref{vanishxi}) and (\ref{vanishDxi}), non-zero expansions in
incomplete gamma functions are not in the kernel of $D^{k-1}\circ
\xi_k$. This implies that
$F_{f,2}-\mathcal{M}_{f, 2}^{+-}$=0.
\end{proof}
\section{A Mock Eichler-Shimura isomorphism}
\label{ESchar}
In this section, we will show an Eichler-Shimura type theorem for
harmonic period functions of positive weight.
We first note that
\begin{equation} \label{xi}
\xi_k(W_{k, 2}) \subset
W_{k-2}\text{,}
\end{equation}
because $\xi_k$ is compatible with the group action of $\Gamma$.

Fix $P \in W_{k, 2}.$ Then \eqref{xi}  and Theorem \ref{ES}
imply that
there exist $f, g \in S_k$ and $a  \in \C$ such that
\begin{equation} \label{eseq}
\xi_{k} (P(z))=r_f(z)+r_{g}(-z)+a\left(z^{k-2}-1\right).
\end{equation}
This can be viewed as a differential equation for $P$, and we will
now  describe the general solution in $W_{k,2}$.
 To find a preimage of the second summand we require regularized
integrals as they are defined, for instance, by Fricke in his upcoming PhD
thesis \cite{Fr}.

Consider a function $f:\field{H}\to\C$ that is continuous.  Assume that there is a $c\in\R^+$ such that
\begin{equation}
f(z)=O\Big(e^{c \, \mathrm{Im}(z)}\Big)
\label{bound}
\end{equation}
uniformly in $\mathrm{Re}(z)$ as  $\text{Im}(z)\to\infty$. Then, for
each $z_0 \in \mathfrak H$, the integral
\[
\int_{z_0}^{i\infty} e^{uw} f(w) \; dw
\]
(where the path of integration lies within a  vertical strip) is
convergent
for $u \in \C$ with $\mathrm{Im}(u) \gg 0$.  If it
has an analytic continuation to $u=0$, we define the {\it regularized
integral}
\[
R.\int_{z_0}^{i\infty}f(w) \; dw:=
\left[\int_{z_0}^{i\infty}e^{uw}f(w) \;dw\right]_{u=0}
\text{,}
\]
where the right hand side means the value at $u=0$ of the analytic continuation
of the integral.
Similarly, we define integrals at other cusps $\mathfrak a$.
Specifically,
suppose that $\mathfrak a=\sigma_{\mathfrak a}(i \infty)$ for a scaling
matrix $\sigma_{\mathfrak a} \in$ SL$_2(\Z)$.  If
$f(\sigma_{\mathfrak a} z)$ satisfies \eqref{bound}, then we define
\[
  R.\int_{z_0}^{\mathfrak a}f(w) \; dw
:=
  R.\int_{\sigma_{\mathfrak a}^{-1} z_0}^{i\infty}f \big|_2\gamma(w) \; dw.
\]
For cusps $\ca, \cb$ we define:
\begin{equation}
\label{reg}
  R.\int_{\ca}^{\cb}f(w) \; dw
:=
  R.\int_{z_0}^{\cb}f(w) \; dw + R.\int_{\ca}^{z_0}f(w) \; dw
\end{equation}
for any $z_0 \in \mathfrak H.$ An easy calculation shows: 
\begin{lemma}
\label{indep}
The integral $R.\int_{\ca}^{\cb}f(w) \,dw$ as defined in \eqref{reg} is
independent of $z_0 \in \mathfrak H.$
\end{lemma}
By Theorem \ref{periodPoincareintro1}, there exists a harmonic Maass
form $M_f$ such that
\begin{equation}\label{bgko}
r_f(-z)=M^+_f\Big|_k(1-S)(z).
\end{equation}
Set
\begin{align*}
  \mathcal{F}_{f, 2}^\ast(z)
&:=
  R.\int_{-\overline{z}}^{i\infty}\frac{M^+_f(w)}{(w+z)^k} \; dw
\text{,}
\\[4pt]
  r_{f, 2}^\ast(z)
&:=
  R. \int_0^{i\infty}\frac{M^+_f(w)}{(w+z)^k} \; dw \, \Big|_k S
\text{,}
\\[4pt]
  \widetilde{r}_{f, 2}^\ast(z)
&:=
  \int_{-\overline{z}}^{i\infty}\frac{r_f(-w)}{(w+z)^k} \; dw
\text{,}
\\[4pt]
  \widehat{r}_{f, 2}^*(z)
&:=
  r_{f, 2}^\ast (z)-\widetilde{r}_{f, 2}^\ast(z)
\text{.}
\end{align*}
We note that, by definition,
$$
M^+_f(z)=\sum_{n=N}^{0}a_n e^{2 \pi i n
z}+O\left(e^{-2 \pi y}\right)
 \quad \text{for some $N<0$, as $y \to \infty$}.
 $$
We insert the above Fourier expansion into $\mathcal{F}_{f,2}^\ast$ and integrate each of the terms separately.  Terms with $n \ge 0$ do not require regularization.  For terms with $n < 0$ we obtain a
linear combination  of incomplete gamma functions of the form $\Gamma(\ell, z)$ ($\ell \in\Z$, $z \not=0$).
These functions can be analytically continued,
from which we may  deduce that
the integrals  can be extended to $u=0$. Therefore, the
regularized integrals are well-defined.  The integral $r_{f, 2}^\ast$ ist treated analogously.

We also  note that
$\widetilde{r}_{f, 2}^\ast$
does not require  regularization, since $r_f(-z)\in V_{k-2}$.
We easily compute, using (\ref{conj}), that
\begin{equation}\label{rstar}
\xi_k\left(\widehat{r}_{f, 2}^*(z)\right)=(2i)^{1-k}r_{f^c}(-z).
\end{equation}

We claim that a special solution in $W_{k, 2}$ to (\ref{xi}) is then
given by
\begin{equation}\label{ssol}
R_{f, 2}^*(z)
:= -(2i)^{k-1} \widehat
r_{f^c, 2}(z) -
(2i)^{k-1} \widehat r^*_{g^c, 2}(z)+
\overline a(2i)^{k-1}
\left ( \int_{-\overline z}^{i \infty} \frac{dw}{(w+z)^k}\right) \Big
|_k(1-S).
\end{equation}
It is clear by \eqref{W_k,2}, \eqref{rstar} and the identity
\begin{equation}\label{third}
\xi_k\left ( \int_{-\bar z}^{i \infty} \frac{dw}{(w+z)^k}
\right)=(2i)^{1-k}
\end{equation}
that $R_{f, 2}^*$ satisfies \eqref{eseq}.

By Theorem \ref{W_k,2}, the function
$\widehat r_{f^c, 2}$ is an element of $W_{k, 2}$.
The same is true for $\widehat{r}_{f, 2}^*$:
 \begin{lemma}
\label{per*}
We have
\[
\mathcal{F}_{f, 2}^\ast\Big|_k(S-1)(z)=\widehat{r}_{f, 2}^*(z).
\]
In particular, $\widehat{r}_{f, 2}^*\in W_{k,2}$.
\end{lemma}
\begin{proof}
We first note, with Lemma \ref{indep} and the definition of
regularized integrals, that
\begin{align}
\nonumber
  r_{f, 2}^\ast|_kS
&=
  \left [
  \int_{-\bar z}^{i \infty} \frac{e^{wu}M^+_f(w) \; dw}
                            {(w+z)^k}
  \right ]_{u =0}
  -
  \left [
  \int_{1/\bar z}^{i \infty} \frac{e^{wu} M^+_f(-1/w) \;d(-1/w)}
                             {(-1/w+z)^k}
  \right ]_{u=0}
\\[4pt]
&=
  \left [
  \int_{-\bar z}^{i \infty} \frac{e^{wu} M^+_f(w) \; dw}
                            {(w+z)^k}
  \right ]_{u=0}
  -
  \left [\int_{-\bar z}^{0} \frac{e^{-u/w} M^+_f(w) \; dw}
                              {(w+z)^k}
  \right ]_{u=0}
\text{.}
\label{r}
\end{align}
On the other hand, to compute
$\mathcal{F}_{f, 2}^\ast |_k(S-1)(z)=\mathcal F_{f,2}^*(-1/z)z^{-k}-\mathcal F_{f, 2}^*(z)$
we recall that, by definition, this is the value of $u$ at $0$ of the analytic continuation of
\begin{gather*}
  \int_{1/\bar z}^{i \infty} \frac{e^{wu}M^+_f(w) \; dw}
                                  {(wz-1)^k}
  -
  \int_{-\bar z}^{i \infty} \frac{e^{w u} M^+_f(w) \; dw}
                                 {(w+z)^k}
\text{.}
\end{gather*}
For $\mathrm{Im}(u) \gg 0$, with \eqref{bgko} this equals
\begin{multline}
  \int_{-\bar z}^{0} \frac{e^{-u/w}M^+_f(-1/w) \; d(-1/w)}
                       {(-z/w-1)^k}
  -
  \int_{-\bar z}^{i \infty} \frac{e^{wu} M^+_f(w) \; dw}
                            {(w+z)^k}
\\=
  \int_{-\bar z}^{0} \frac{e^{-u/w}M^+_f(w) \; dw}
                       {(z+w)^k}
  -
  \int_{-\bar z}^{0} \frac{e^{-u/w}r_f(-w) \; dw}
                       {(z+w)^k}
  -
  \int_{-\bar z}^{i \infty} \frac{e^{wu} M^+_f(w) \; dw}
                            {(w+z)^k }
\text{.}
\label{compu}
\end{multline}
Because of \eqref{periodrel}, the second integral of \eqref{compu} equals
$$\int_{-1/\bar z}^{i \infty} \frac{e^{wu}r_f(-1/w) w^k \; dw}{(zw-1)^k}
=-\int_{1/\bar z}^{i \infty}
\frac{e^{wu}r_f(w) \; dw}{(zw-1)^k}
$$
This is analytic at $u=0$ with value $\tilde r^*_{f,2}|_kS(z)$. Therefore,
with analytic continuation
and \eqref{r}, \eqref{compu} gives
\begin{gather*}
  \mathcal F_{f, 2}^*|_k(S-1)
=
  -r_{f, 2}^*|_k S + \widetilde{r}_{f, 2}^* \big|_k S
=
  -\widehat{r}_{f, 2}^* \big|_k S
\text{,}
\end{gather*}
which implies the result.
\end{proof}
That the  third term of \eqref{ssol} is an element of $W_{k, 2}$
follows directly  from \eqref{third} and the invariance of the integral
under $T$.

Therefore, the general solution of \eqref{eseq} is
\begin{equation*}\label{gensolution}
-(2i)^{k-1} \left
(\widehat
r_{f^c, 2}(z) + \widehat r^*_{g^c, 2}(z)-\overline a
 \int_{-\overline z}^{i \infty} \frac{dw}{(w+z)^k} \Big
|_k(1-S)+G(z) \right),
\end{equation*}
where $G$ is a holomorphic function on $\mathfrak H$.
The last summand $G$ must be annihilated by $1+S$ and $1+U+U^2$ in terms of $|_k$, because all the others satisfy the Eichler-Shimura relations. This implies
that $G=H|_k(S-1)$  for some  translation invariant holomorphic
function $H$. Indeed, this follows from $H^1(\G, \mathcal A)=0$, where
$\mathcal A$ is a the module
of holomorphic functions on $\mathfrak H$
(see equation (5.3) of \cite{Kn}  citing \cite{Kr}).

Set
$$
U_{k, 2}:=\Big (
\mathcal O(\mathfrak H)+\left\{f \in \oplus_{j=1}^{k-1}y^{-j}V_{k-2};\,
\xi_k(f) \in V_{k-2} \right\}
\Big ) \cap \{f: \HH \to \C; f|_kT=f\},
$$
where $\mathcal O(\mathfrak H)$ is the space of holomorphic functions
on $\mathfrak H$.
We can then complete the proof of
\begin{theorem} The map $\phi: S_k \oplus S_k \to W_{k, 2}$
defined by
$$
\phi(f,g):=\widehat r_{f^c, 2}+\widehat r^*_{g^c, 2}
$$
induces an isomorphism
$$
\overline{\phi}:\,  S_k \oplus S_k \cong_{\mathbb R} W_{k, 2}/V_{k, 2},
$$
where $V_{k,   2}:=U_{ k, 2}|_k(S-1)$.
\end{theorem}
\begin{proof} We have already shown above that
$\overline{\phi}$  is surjective.
To show that it is injective,
suppose that $P \in \mathrm{ker}(\overline{\phi})$.
Then
\begin{equation} \label{ESlike}
\widehat r_{f^c, 2} + \widehat r^*_{g^c, 2} = A|_k(S-1)
\end{equation}
for some $A \in U_{k, 2}.$
 Applying $\xi_k$ on both sides  of \eqref{ESlike},  we
deduce that $r_f(z)+r_g(-z)$ is an Eichler coboundary. The classical
Eichler-Shimura isomorphism  (Theorem \ref{ES}) \ implies that $f, g$
must vanish.
\end{proof}
\begin{remark}
 Since $\left\{f \in \oplus_{j=1}^{k-1}y^{-j}V_{k-2};\,
\xi_k(f) \in V_{k-2} \right\}$ does not contain any holomorphic elements, it is isomorphic to $V_{k -2}$.  The corresponding isomorphism is $\xi_k$.
\end{remark}

\section{Cohomological interpretation}\label{cohom}
Theorem \ref{periodPoincare} has a cohomological interpretation which
makes apparent the
similarity of our construction with the one associated to critical values
in \cite{BGKO}. We shall first give a cohomological interpretation of the
period polynomials in the context of the results of \cite{BGKO}.

We recall the definition of parabolic cohomology in our  setting.
For  $m \in \Z$ and  a $\G$\nobreakdash-\hspace{0pt}submodule
$V$ of the space of functions $f: \HH \to \C$ we  define
\begin{align*}
  Z^1_p(\G, V)
&:=
  \bigl\{g: \G\to V; g(\g\delta)=g(\g)|_m\delta+g(\delta)
  \text{ and}
\\&\hspace{7em}
  g(T)=h|_m(T-1) \text{ for some } h \in V \bigr\},
\end{align*}
\begin{align*}
  B^1_p(\G, V)=B^1(\G, V)
&:=
  \bigl\{g: \G\to V; \text{for some } h \in V,
\\&\hspace{7em}
  g(\g)=h|_m(\g-1) \text{ for all }\g \in \G \bigr\},
\end{align*}
and
$$
H^1_p(\G, V):=Z^1_p(\G, V)/B^1_p(\G, V).
$$
A basic map in the theory of period polynomials is
$$
\rho: S_k \to H^1_p(\G, V_{k-2}).
$$
It assigns to $f \in S_k$ the class of a cocycle $\phi_f$ determined
by $\phi_f(T)=0$ and $\phi_f(S)=r_f(-z)$.
 We further consider the $\G$-module
$\mathcal O^*(\mathfrak H)$ of holomorphic functions
$F: \mathfrak H \to \C$ of at most linear exponential
growth at the cusps.
The group $\G$ acts on $\mathcal O^*(\mathfrak H)$ via $|_{2-k}$.
Then the natural injection $i$ of $V_{k-2}$ into
$\mathcal O^*(\mathfrak H)$
induces a map
$$
i^*: H^1_p\left(\G, V_{k-2}\right) \to H^1_p\left(\G, \mathcal
O^*(\mathfrak H)\right).
$$
Theorem 1.1 of \cite{BGKO} states that $r_f(-z)$ is a constant
multiple of
$ F_f^+|_{2-k}(1-S)$ for the holomorphic part $F_f^+$ of
some harmonic Maass form $F_f$ that grows at most linear
exponentially at the cusps.
This can then be reformulated as:
\begin{equation}\label{ESBGKO}i^* \circ \rho=0.
\end{equation}

To
formulate the analogue of this result in our
context and the setting of non-critical values we consider the following
$\G$-modules, all in terms of the action $|_k$,
\begin{enumerate}
\item[i)] $\HHH^*(\HH)$ the $\G$-module of harmonic
functions on $\HH$ of at most linear exponential growth at the
cusps.
\item[ii)]
$\mathcal V_{k, 2}:=\{f: \HH \to \C
 \text{ of at most lin. exp. growth at the cusps,
}
\xi_k(f) \in V_{k-2} \}$.
\end{enumerate}
Because of
the compatibility of $\xi_{k}$ with the
slash action, these spaces are $\G$-invariant.

According to Theorem \ref{W_k,2}, for each $f \in S_k$, the map
$\psi_f$ such that $\psi_f(T)=0$ and $\psi_f(S)=
\widehat{r}_{f, 2}$
induces a cocycle with values in $\mathcal V_{k, 2}$.
Therefore, the assignment $f \to \psi_f$ induces a linear
map
$$
\rho': S_k \to H^1_p\left(\G, \mathcal V_{k, 2}\right).
$$
Because of Remark \ref{harmo}, there is a natural injection $i'$ from
$\mathcal V_{k, 2}$ to $\HHH^*(\HH)$, and this induces a map:
$$
i'^*:
H^1_p \left(\G, \mathcal V_{k, 2}\right) \to
H^1_p \left(\G, \HHH^*(\HH)\right).
$$

 Theorem \ref{periodPoincare} then implies that
\begin{theorem} \label{periodPoincare'} The composition
$i'^* \circ \rho'$ is the zero map.
\end{theorem}

\end{document}